\documentclass[12pt]{article}

\usepackage{amsthm,amssymb,amsmath}
\usepackage{epsfig,fullpage}

\title{Median eigenvalues of bipartite subcubic graphs}

\author{Bojan Mohar\thanks{Supported in part by the
  Research Grant P1--0297 of ARRS (Slovenia), by an NSERC Discovery Grant (Canada)
  and by the Canada Research Chair program.}~\thanks{On leave from:
  IMFM \& FMF, Department of Mathematics, University of Ljubljana, Ljubljana,
  Slovenia.}\\
  Department of Mathematics, Simon Fraser University,\\
  Burnaby, B.C. V5A 1S6 \\
  email: {\tt mohar@sfu.ca}
}
\date{}

\newtheorem{theorem}{Theorem}[section]
\newtheorem{lemma}[theorem]{Lemma}
\newtheorem{corollary}[theorem]{Corollary}
\newtheorem{proposition}[theorem]{Proposition}
\newtheorem{conjecture}[theorem]{Conjecture}

\newcommand{\DEF}[1]{{\em #1\/}}
\renewcommand\l{\lambda}
\newcommand\ball[1]{B_{#1}(v_0)}
\newcommand\ve{\varepsilon}
\begin{document}

\maketitle

\begin{abstract}
It is proved that the median eigenvalues of every connected bipartite graph $G$ of maximum degree at most three belong to the interval $[-1,1]$ with a single exception of the Heawood graph, whose median eigenvalues are $\pm\sqrt{2}$. Moreover, if $G$ is not isomorphic to the Heawood graph, then a positive fraction of its median eigenvalues lie in the interval $[-1,1]$.
This surprising result has been motivated by the problem about HOMO-LUMO separation that arises in mathematical chemistry.
\end{abstract}


\section{Introduction}

In a recent work, Fowler and Pisanski \cite{FP1,FP2} introduced the notion of the \emph{HL-index} of a graph that is related to the HOMO-LUMO separation studied in theoretical chemistry (see also Jakli\v c, Fowler, and Pisanski \cite{FJP}). This is the gap between the Highest Occupied Molecular Orbital (HOMO) and Lowest Unoccupied Molecular Orbital (LUMO). In the H\"uckel model \cite{GuPol}, the energies of these orbitals are in linear relationship with eigenvalues of the corresponding molecular graph and can be expressed as follows. Let $G$ be a graph of order $n$, and let
$\lambda_1(G)\ge \lambda_2(G)\ge \cdots \ge\lambda_n(G)$ be the eigenvalues of its adjacency matrix. The eigenvalues occurring in the HOMO-LUMO separation are the median eigenvalues $\lambda_h(G)$ and $\lambda_\ell(G)$, where
$$
   h=\lfloor \tfrac{n+1}{2}\rfloor \qquad \textrm{and} \qquad
   \ell=\lceil \tfrac{n+1}{2}\rceil.
$$
The \emph{HL-index\/} $R(G)$ of the graph $G$ is then defined as
$$R(G) = \max \{|\lambda_h(G)|,|\lambda_\ell(G)|\}.$$

A simple unweighted graph $G$ is said to be \DEF{subcubic} if its maximum degree is at most $3$. In \cite{FP1,FP2} it is proved that every subcubic graph $G$ satisfies $0\le R(G) \le 3$ and that if $G$ is bipartite, then $R(G)\le \sqrt{3}$. The following is the main result from \cite{Moh}.

\begin{theorem}[Mohar \cite{Moh}]
\label{thm:cubic}
The median eigenvalues $\l_h(G)$ and $\l_\ell(G)$ of every subcubic graph\/ $G$ are contained in the interval $[-\sqrt{2},\sqrt{2}\,]$, i.e., $R(G) \le \sqrt{2}$.
\end{theorem}

This result is best possible since the Heawood graph (the bipartite incidence graph of points and lines of the Fano plane) has $\l_h=-\l_\ell=\sqrt{2}$.

The following conjecture was proposed in \cite{Moh}.

\begin{conjecture}
\label{conj:planar}
If\/ $G$ is a planar subcubic graph, then $R(G)\le1$.
\end{conjecture}

The conjecture has been verified for planar bipartite graphs in \cite{Moh2}.
In this paper we prove a surprising extension of \cite{Moh2} and of Conjecture \ref{conj:planar} that holds for all bipartite subcubic graphs with a single exception of the Heawood graph (or disjoint union of copies of it).
The following are our main results.

\begin{theorem}
\label{thm:cubic bipartite}
Let $G$ be a bipartite subcubic graph. If every connected component of $G$ is isomorphic to the Heawood graph, then $R(G)=\l_h(G)=|\l_\ell(G)|=\sqrt{2}$. In any other case, the median eigenvalues $\l_h(G)$ and $\l_\ell(G)$ are contained in the interval $[-1,1]$, i.e., $R(G) \le 1$.
\end{theorem}

Theorem \ref{thm:cubic bipartite} shows that the median eigenvalues $\l_h$ and $\l_\ell$ are small, but our proof can be tweaked to give much more -- a positive fraction of (median) eigenvalues lie in the interval $[-1,1]$.

\begin{theorem}
\label{thm:cubic bipartite linearly many}
There is a constant $\delta>0$ such that for every bipartite subcubic graph $G$, none of whose connected components is isomorphic to the Heawood graph, all its eigenvalues $\l_i(G)$, where $(\tfrac{1}{2}-\delta)n\le i\le (\tfrac{1}{2}+\delta)n$, belong to the interval $[-1,1]$.
\end{theorem}

\section{Interlacing and imbalance of partitions}

Let us first recall that eigenvalues of bipartite graphs are symmetric with respect to 0, i.e., if $\l$ is an eigenvalue, then $-\l$ is an eigenvalue as well and has the same multiplicity as $\l$. This in particular implies that $\l_h\ge0$ and that $\l_\ell=-\l_h$. Therefore, it suffices to consider $\l_h$.

Let us next recall the eigenvalue interlacing theorem (cf., e.g., \cite{GR}) that will be our main tool in the sequel.

\begin{theorem}
\label{thm:interlacing}
Let\/ $A\subset V(G)$ be a vertex set of cardinality $k$, and let $K=G-A$. Then for every $i=1,\dots,n-k$, we have
$$
    \l_i(G)\ge \l_i(K)\ge \l_{i+k}(G).
$$
\end{theorem}

If $V(G)=A\cup B$ is a partition of the vertices of $G$, we denote by $G(B)=G-A$ the subgraph of $G$ induced on $B$.
%
%
In the sequel we will consider vertex partitions $V(G)=A\cup B$, but the two parts $A,B$ will play different roles. Thus, we shall consider such a partition as an ordered pair $(A,B)$. Given a partition $(A,B)$ of $V(G)$, let $s\ge1$ be the smallest integer such that $\l_s(G(B))\le1$, and let $t=\left\lfloor \tfrac{1}{2}(|B|-|A|+1)\right\rfloor$. Then we say that the partition $(A,B)$ is \DEF{$(s,t)$-imbalanced}, and we define the \DEF{imbalance} of the partition $(A,B)$ as $$imb(A,B)=t-s+1.$$

\begin{lemma}
\label{lem:k-unbalanced}
Suppose that\/ $(A,B)$ is an $(s,t)$-imbalanced vertex partition of a subcubic graph $G$. If $r=imb(A,B)-1$, then $\l_{h-r}(G)\le1$. In particular, if $imb(A,B)>0$, then $\l_h\le1$.
\end{lemma}

\begin{proof}
Conditions of the lemma give that $\l_s(G(B))\le1$ and
$$
   r=\left\lfloor \tfrac{1}{2}(|B|-|A|+1)\right\rfloor - s.
$$
If $n=|A|+|B|$ is even, then $r=\tfrac{1}{2}|B|-\tfrac{1}{2}|A|- s = \tfrac{n}{2}-(|A|+s)$. If $n$ is odd, then $r=\tfrac{n+1}{2}-(|A|+s)$. In each case,
$$
   |A|+s = \left\lfloor\tfrac{n+1}{2}\right\rfloor - r = h-r.
$$
Since $G(B)=G-A$ is obtained from $G$ by deleting $|A|$ vertices and $|A|+s\le h$, the eigenvalue interlacing theorem shows that $\l_{h-r}(G) = \l_{|A|+s}(G)\le \l_s(G(B))\le 1$.
\end{proof}

Let $(A,B)$ be a partition of $V(G)$. Suppose that $C\subseteq A$ is a set of vertices in $A$.
We say that $C$ \DEF{increases imbalance} of $(A,B)$ 
if $imb(A\setminus C, B\cup C) > imb(A,B)$. The following result will be our main tool for finding imbalance-increasing vertex-sets.

\begin{lemma}
\label{lem:increase imbalance}
Suppose that\/ $(A,B)$ is a partition of\/ $V(G)$ and $C\subseteq A$.
Let $Q$ be the union of those connected components of $G(B\cup C)$ that contain vertices in $C$.
If $\l_{|C|}(Q)\le1$, then $C$ increases imbalance of $(A,B)$.
\end{lemma}

\begin{proof}
Let $(A,B)$ be $(s,t)$-imbalanced. Let $k=|C|$ and $(A',B')=(A\setminus C, B\cup C)$. Then $t':= t+k = \left\lfloor \tfrac{1}{2}(|B'|-|A'|+1)\right\rfloor$. Note that $G(B)$ and $G(B')$ have the same connected components except for those contained in $Q$. Since $\l_k(Q)\le 1$ and $\l_s(G(B))\le 1$, we have that $\l_{s+k-1}(G(B'))\le1$. Thus $(A',B')$ is $(s',t')$-imbalanced, where $s'\le s+k-1$. Hence,
$imb(A',B') \ge (t+k) - (s+k-1) + 1 = imb(A,B) + 1$.
\end{proof}

For $C\subseteq V(G)$, let $N(C)$ denote the set of all vertices in $V(G)\setminus C$ that have a neighbor in $C$. The statement of Lemma \ref{lem:increase imbalance} has a converse under a mild restriction on $N(C)$.

\begin{lemma}
\label{lem:increase imbalance converse}
Suppose that\/ $(A,B)$ is a partition of\/ $V(G)$ and $C\subseteq A$. If\/ $N(C)\subseteq B$ and every vertex in $N(C)$ has all its neighbors in $A$, then $C$ increases imbalance of $(A,B)$ if and only if $\l_{|C|}(Q)\le1$, where $Q=G(C\cup N(C))$.
\end{lemma}

\begin{proof}
Let us first observe that the independence of $N(C)$ in $G(B)$ implies that the subgraph $Q$ that appears in Lemma \ref{lem:increase imbalance} is equal to $G(C\cup N(C))$. Therefore, it remains to prove only the direction converse to the one of the previous lemma, i.e., that the increased imbalance implies that $\l_{|C|}(Q)\le1$.

Let us adopt the notation used in the proof of Lemma \ref{lem:increase imbalance} and let $s'$ be the smallest integer such that $\l_{s'}(G(B'))\le1$. Let us assume that $imb(A',B') = t'-s'+1 > imb(A,B) = t-s+1$. Since $t'=t+k$, it suffices to see that $s'=s+k-1$. However, this is an easy observation since $G(B)$ and $G(B')$ have the same eigenvalues apart from the eigenvalues of $Q$ in $G(B')$ that are replaced by $|N(C)|$ eigenvalues, all equal to 0, in $G(B)$.
\end{proof}

Suppose now that $G$ is a bipartite graph and that $(A,B)$ is its bipartition.
A set $U$ of vertices of $G$ is \DEF{$A$-thick} (\DEF{$B$-thick}) in $G$ if every vertex in $U\cap A$ (resp.\ $U\cap B$) has at most one neighbor in $V(G)\setminus U$, every vertex in $B\setminus U$ (resp.\ $A\setminus U$) has at most one neighbor in $U$, and $|U\cap A|>|U\cap B|$ (resp.\ $|U\cap A|<|U\cap B|$). The set $U$ is \DEF{thick} if it is either $A$-thick or $B$-thick.

\begin{lemma}
\label{lem:thick unbalanced}
If\/ $U$ is an $A$-thick set of vertices in $G$, then the set $C=A\cap U$ increases imbalance of the bipartition $(A,B)$.
\end{lemma}

\begin{proof}
Consider the subgraph $Q$ of $G(B\cup C)$ consisting of those connected components that contain vertices in $C$, and let $t=|B\cap U|<|C|$. Thickness condition implies that after removing vertices in $B\cap U$ from $Q$, we are left with a graph consisting of a matching and isolated vertices, so its eigenvalues are all in the interval $[-1,1]$. By the eigenvalue interlacing theorem, we conclude that $\l_{|C|}(Q) \le \l_{t+1}(Q) \le \l_1(Q-(B\cap U)) \le 1$, so $C$ increases imbalance of $(A,B)$ by Lemma \ref{lem:increase imbalance converse}.
\end{proof}

\section{Improving imbalance}

\label{sect:improving imbalance}

In this section we prove that for every vertex $v_0$ of a connected bipartite subcubic graph $G$ with bipartition $(A,B)$, if $G$ is not isomorphic to the Heawood graph, then a small neighborhood around $v_0$ contains a set $C$ that can be used to increase imbalance of the partition $(A,B)$ or $(B,A)$. From now on we assume that $G$ is bipartite and $(A,B)$ is the bipartition of $G$.

Given a graph $G$ and its vertex $v$, we denote by $B_r(v)$ the set of all vertices of $G$ whose distance from $v$ is at most $r$. We will sometimes consider the set $B_r(v)$ as the subgraph of $G$ induced on this vertex set.

\begin{lemma}
\label{lem:improve imbalance}
Suppose that\/ $(A,B)$ is the bipartition of a bipartite subcubic graph $G$, $v_0\in V(G)$, and the connected component of $G$ containing $v_0$ is not isomorphic to the Heawood graph. Then $\ball{17}$ contains a set $C$ of vertices such that either $C\subseteq A$ and $C$ increases imbalance of $(A,B)$, or
$C\subseteq B$ and $C$ increases imbalance of $(B,A)$.
\end{lemma}

Before giving the proof of Lemma \ref{lem:improve imbalance}, let us show how the lemma implies our main results, Theorems \ref{thm:cubic bipartite} and \ref{thm:cubic bipartite linearly many}.

\begin{proof}[Proof of Theorems \ref{thm:cubic bipartite} and \ref{thm:cubic bipartite linearly many}]
The proof of Theorem \ref{thm:cubic bipartite} follows from the proof of Theorem \ref{thm:cubic bipartite linearly many} given below by taking $V_0=\{v_0\}$, where $v_0$ is an arbitrary vertex of $G$. Thus we only need to take care of Theorem \ref{thm:cubic bipartite linearly many}.

Let $G$ be a bipartite subcubic graph with no component isomorphic to the Heawood graph. For each vertex $v$ of $G$, we have $|B_r(v)|< 3\cdot 2^r$. Therefore, $G$ contains a set $V_0$ of vertices that are mutually at distance at least 38 and $|V_0|>\ve n$, where $\ve=2^{-40}$.
Let us consider the bipartition $(A,B)$ of $G$, and let, for each $v\in V_0$, $C_v$ be the vertex set $C$ obtained by applying Lemma \ref{lem:improve imbalance}. Let $a$ denote the number of vertices $v\in V_0$ such that $C_v\subseteq A$, and let $b=|V_0|-a$ be the number of cases where $C_v\subseteq B$.

Note that $\l_1(G(A))=\l_1(G(B))=0$. Thus,
\begin{equation}
\label{eq:pf1}
imb(A,B) + imb(B,A) = \lfloor \tfrac{1}{2}(|B|-|A|+1)\rfloor + \lfloor \tfrac{1}{2}(|A|-|B|+1)\rfloor \ge 0.
\end{equation}
Let $(A',B')$ be the partition obtained from $(A,B)$ by removing from $A$ and adding into $B$ all sets $C_v$ ($v\in V_0$) for which $C_v\subseteq A$. Since every $C_v$ is contained in $B_{17}(v)$, all these sets are pairwise at distance at least 4 from each other. Consequently, their graphs $Q=Q_{C_v}$ are pairwise disjoint and non-adjacent. Hence, each of these sets increases imbalance of $(A,B)$ by at least 1 (Lemma \ref{lem:increase imbalance converse}). Therefore,
\begin{equation}
\label{eq:pf2}
   imb(A',B')\ge imb(A,B) + a.
\end{equation}
Similarly, adding to $A$ all sets $C_v$ ($v\in V_0$) for which $C_v\subseteq B$, we obtain a partition $(B'',A'')$ such that
\begin{equation}
\label{eq:pf3}
   imb(B'',A'')\ge imb(B,A) + b.
\end{equation}
Finally, all three inequalities (\ref{eq:pf1})--(\ref{eq:pf3}) imply that
\begin{equation}
\label{eq:pf4}
imb(A',B') + imb(B'',A'') \ge imb(A,B) + imb(B,A) + a + b > \ve n.
\end{equation}
By symmetry, we may assume that $imb(A',B') \ge imb(B'',A'')$.
Then (\ref{eq:pf4}) implies that $imb(A',B')\ge \tfrac{1}{2}\ve n$, and Lemma \ref{lem:k-unbalanced} gives the claim of the theorem with $\delta=\tfrac{1}{2}\ve$.
\end{proof}

The family ${\mathcal C}_0$ of graphs listed in the appendix (Figures \ref{fig:appendix1}--\ref{fig:appendix3}) has the following property: If $H\in {\mathcal C}_0$ and $C$ is the bipartite set of its vertices that are shown as full circles or full squares in Figures \ref{fig:appendix1}--\ref{fig:appendix3}, then $\l_{|C|}(H)\le1$. Lemma \ref{lem:increase imbalance converse} shows that the following is the common outcome for all of these graphs:

\begin{corollary}
\label{cor:appendix}
Suppose that\/ $H\in{\mathcal C}_0$ is one of the graphs depicted in Figures \ref{fig:appendix1}--\ref{fig:appendix3}. Let $C$ be the bipartite class of its vertices that are drawn as full circles or full squares. If a bipartite subcubic graph $G$ with bipartition $(A,B)$ contains $H$ as an induced subgraph, where $C\subseteq A$, and every vertex in $C$ has all its neighbors in $H$, then $C$ increases imbalance of the bipartition $(A,B)$ of $G$.
\end{corollary}

\begin{proof}
The proof is clear by observing that the component $Q$ in $G(B\cup C)$ containing $C$ is equal to $H$ and by the remarks given in the paragraph before the corollary.
\end{proof}

We shall need some new concepts. Let $k\ge1$ be an integer. We say that vertices $x$ and $y$ of $G$ are \DEF{$k$-adjacent} in $G$ if there is a path of length $k$ joining them. If $H$ is a subgraph of $G$, a \DEF{$k$-chord} of $H$ in $G$ is a path $P=v_0v_1\dots v_k$ of length $k$, such that $P\cap H=\{v_0,v_k\}$. Having such a $k$-chord $P$, we say that $v_0$ and $v_k$ are \DEF{$k$-adjacent outside} $H$. The subgraph $H$ is \DEF{$k$-induced} in $G$ if it has no $l$-chords for $l=1,\dots,k$. Note that the special case when $k=1$ gives the usual notions of being adjacent, a chord, or an induced subgraph.

When dealing with vertices of degree 2 in the proof of Lemma \ref{lem:improve imbalance}, we will need the following result.

\begin{proposition}
\label{prop:path in the ball}
Let\/ $G$ be a bipartite graph of girth at least 6, let $v_0\in V(G)$, and let $r$ be a positive integer.
If $x,y\in \ball{r}$, then there exists a 2-induced path from $x$ to $y$ that is contained in $\ball{r+1}$.
\end{proposition}

\begin{proof}
For $u,v\in V(G)$, we will denote by $d(u,v)$ the distance in $G$ from $u$ to $v$.
Let $P_x$ and $P_y$ be shortest paths from $x$ and from $y$ to $v_0$, respectively. Choose these paths so that their intersection is a path $Q$, and $Q$ is as long as possible. Since $G$ has no cycles of length 4, both paths $P_x$ and $P_y$ are 2-induced. Let $z$ be the first vertex of intersection of $P_x$ and $P_y$ when traversing the paths in the direction towards $v_0$. Then $Q$ is a path from $z$ to $v_0$. The path $P_{xy}$ consisting of the segments of $P_x$ from $x$ to $z$ and of $P_y$ from $z$ to $y$ has length $l=d(x,v_0)+d(y,v_0)-2d(z,v_0)$. First, we claim that $P_{xy}$ is an induced path. Namely, the segments of the two paths are induced, so if there were a chord $uv$ of $P_{xy}$, then $u\in V(P_x)\setminus V(P_y)$ and $v\in V(P_y)\setminus V(P_x)$. Since $G$ is bipartite, we have $d(u,z)\ne d(v,z)$, and it is easy to see that this contradicts our choice of the paths with $Q$ being longest possible since we could replace one of the paths by a path using the edge $uv$ and thus increasing the intersection of the two paths.

If $P_{x,y}$ is not 2-induced, let $uwv$ be a 2-chord, where $u\in V(P_x)\setminus V(P_y)$ and $v\in V(P_y)\setminus V(P_x)$. Let us choose the 2-chord such that $d(u,z)$ is maximum possible. As before, the maximality of $Q$ shows that $d(u,z) = d(v,z)$. Let $P_{xy}'$ be the path from $x$ to $y$ obtained from $P_{xy}$ by using the 2-chord $uwv$ instead of the path from $u$ to $v$ in $P_{xy}$. It is clear from our choices that any chords or 2-chords of $P_{xy}'$ must use the vertex $w$. However, since $G$ has no 4-cycles, any such chord or 2-chord would give a contradiction to the maximality property of $Q$. Therefore, $P_{xy}'$ is 2-induced.

The conclusion from the above paragraph is that either $P_{xy}$ is 2-induced, or $P_{xy}'$ exists and is 2-induced. In each case we obtain the statement of the proposition.
\end{proof}

\begin{proof}[Proof of Lemma \ref{lem:improve imbalance}]
In the proof, we do not intend to optimize the distance from $v_0$ in which we are able to find a set that increases imbalance. Our main aim is to keep the proof simple.

Suppose that $G$ and its vertex $v_0$ give a counterexample to the theorem. We may assume that $G$ is connected. As before, we assume that $(A,B)$ is the bipartition of $G$.
We shall proceed through a sequence of claims, concerning vertices in vicinity of $v_0$. In each claim we will assume that the claim is false and then define certain vertex set $C$ (where $C\subseteq A$ or $C\subseteq B$). We let $Q_C=G(C\cup N(C))$. Observe that $Q_C$ is the subgraph $Q$ that appears in Lemmas \ref{lem:increase imbalance} and \ref{lem:increase imbalance converse} about increasing imbalance.

\medskip

\emph{Claim 1: Each vertex in $\ball{17}$ has degree at least 2.}
If $v$ has degree at most 1, then $C=\{v\}$ increases imbalance since $\l_1(G(B\cup\{v\}))\le1$ and $\l_1(G(A\cup\{v\}))\le1$.

\medskip

\emph{Claim 2: $\ball{16}$ contains no 4-cycles.}
Suppose that $D=v_1v_2v_3v_4$ is a 4-cycle in $\ball{16}$. For $i=1,\dots,4$, let $u_i$ be the neighbor of $v_i$ that is not in $V(D)$ (if $\deg(v_i)=2$, then we set $u_i=v_i$).
If $u_1\ne u_3$, then let $C=\{v_1,v_3\}$. We may assume that $C\subseteq A$. Clearly, $Q_C = G(C\cup N(C))$ is isomorphic to the graph $C_4^+$ depicted in Figure \ref{fig:appendix1} (or to an induced subgraph of $C_4^+$ if $u_1=v_1$ or $u_3=v_3$). Corollary \ref{cor:appendix} shows that $C$ increases imbalance of $(A,B)$. Similarly, if $u_2\ne u_4$, we may take $C=\{v_2,v_4\}$ and $C$ increases imbalance of $(B,A)$. Finally, assume that $u_1=u_3$ and $u_2=u_4$. Let us now consider the 4-cycle $v_1v_2v_3u_1$. Again, if $u_1$ and $v_2$ have no common neighbor outside this cycle, we are done by taking $C=\{u_1,v_2\}$. (Note that $C\subseteq \ball{17}$ since $V(D)\subseteq \ball{16}$.) If they have such a common neighbor, this must be $u_2$, and hence $G=K_{3,3}$. In this case, $C=\{v_1,v_3,u_2\}$ increases imbalance.

\medskip

\emph{Claim 3: If $\ball{13}$ contains a vertex of degree 2, then $\ball{16}$ contains precisely two vertices of degree 2 and they are adjacent to each other.}
Let us first prove that $\ball{16}$ contains at most two vertices of degree 2, and if there are two, one of them is in $A$ and the other one is in $B$. To see this, suppose that $u,v\in A\cap\ball{16}$ have degree 2 and $u\ne v$. By Proposition \ref{prop:path in the ball}, there exists a 2-induced path $P$ from $u$ to $v$ in $\ball{17}$. Let $C=V(P)\cap A$. Since $P$ is 2-induced and $u,v$ have degree 2, the graph $Q_C$ is isomorphic to the graph $\widehat{P}_{2t+1}$ shown in Figure \ref{fig:appendix1}, where the horizontal path shown at the bottom of the drawing is $P$ and $2t$ is its length. Since $P\subseteq \ball{17}$, Corollary \ref{cor:appendix} completes the proof.

Suppose now that $\deg(v)=2$, where $v\in \ball{13}$. Suppose first that $v$ does not belong to a cycle of length 6. Let $v_1,v_1'$ be the neighbors of $v$, and let $v_2$ ($v_2'$) be a neighbor of $v_1$ ($v_1'$) that is different from $v$. Finally, let $C = \{v,v_2,v_2'\}\subset \ball{15}$. Let $Q=Q_C$. We claim that $Q$ is isomorphic to the graph $\widehat{P}_7^-$ depicted in Figure \ref{fig:appendix1} (since $\deg(v_2)=\deg(v_2')=3$). To see this, we have to prove that vertices in $C\cup N(C)$ are distinct and non-adjacent, apart from their adjacencies in $P_7^-$. Clearly, $v$ is at distance at most 3 from all vertices in $Q$. If two of them were adjacent or the same (apart from adjacencies in $P_7^-$), then we would obtain a cycle of length at most 7 containing $v$. Since $G$ is bipartite and $v$ does not belong to a cycle of length 4 or 6, this is not possible. This proves the claim. Now, we are done by Corollary \ref{cor:appendix}.

Finally, let $R=vv_1v_2v_3v_4v_5$ be a 6-cycle containing $v$. As shown above, we may assume that $\deg(v_2) = \deg(v_4) = 3$.
By symmetry, we may also assume that $\deg(v_1)=3$. It suffices to prove that $\deg(v_5)=2$. Suppose for a contradiction that $\deg(v_5)=3$. For $i\in\{1,2,4,5\}$, let $u_i$ be the neighbor of $v_i$ that is not in $V(R)$. In the preceding paragraph we proved that taking the set $C = \{v,v_2,v_4\}$ gives the outcome of the theorem unless $v_2$ and $v_4$ are 2-adjacent (which in turn gave rise to the 6-cycle $R$). The same argument can be repeated on the sets $\{v,v_2,u_5\}$, $\{v,u_1,v_4\}$, and $\{v,u_1,u_5\}$. They show that $u_2$ is the common neighbor of $v_2$ and $u_5$, $u_4$ is the common neighbor of $u_1$ and $v_4$, and that $u_1$ and $u_5$ have a common neighbor $w\notin V(R)\cup \{u_2,u_4\}$. (Here we used the fact that there are no 4-cycles in $\ball{16}$ and that all treated vertices are in $\ball{16}$.)
If the subgraph $S$ of $G$ induced on the vertex set $V(R)\cup \{u_1,u_2,u_4,u_5,w\}$ is 2-induced, then it is thick and we are done by Lemma \ref{lem:thick unbalanced}. Therefore, two of its vertices are 2-adjacent. The only pairs that could be 2-adjacent outside $S$ without creating a 4-cycle are $u_2,u_4$ and $v_3,w$. Both of these give isomorphic outcomes since $v_3$ and $w$ would have played the role of $u_2$ and $u_4$ if taking the 6-cycle $vv_1v_2u_2u_5v_5v$ instead of $R$. Thus, we may assume that $u_2zu_4$ is a 2-chord.
Let $C = \{v,v_2,v_4,u_1,u_5,z\}$. Observe that $C\subset \ball{15}$ since a neighbor of $v$ belongs to $\ball{12}$. Then $Q_C=G(C\cup N(C))$ is isomorphic to the graph $\widehat{C}_6$ shown in Figure \ref{fig:appendix1}. Thus, we obtain the outcome of the theorem by Corollary \ref{cor:appendix}. This completes the proof of Claim 3.

\medskip

\emph{Claim 3A. If $z_1\in\ball{13}$ and $z_2\in\ball{14}$ are adjacent degree-2 vertices of $G$, then they are contained in a 6-cycle $R=z_1v_1v_2v_3v_4z_2\subset\ball{15}$ and there are vertices $u_1,u_4\notin V(R)$ such that $R' = u_1v_1v_2v_3v_4u_4\subset\ball{15}$ is also a 6-cycle in $G$.}
We shall use the notation used in the last part of the proof of Claim 3, except that we now assume that $z_1=v$ and $z_2=v_5$ both have degree 2. Considering the set $C = \{z_1,u_1,v_4\}$, we conclude that $u_1$ is adjacent to $u_4$. This gives rise to $R'$. Since a neighbor of $z_1$ belongs to $\ball{12}$, we conclude that $R\cup R'\subset \ball{15}$.

\medskip

An induced $2t$-cycle $D=v_1v_2\dots v_{2t}$ in $G$, in which either no two vertices in the set $C=\{v_1,v_3,v_5,\dots,v_{2t-1}\}$ are 2-adjacent outside $D$, or no two vertices in the set $C'=\{v_2,v_4,v_6,\dots,v_{2t}\}$ are 2-adjacent outside $D$, is called a \DEF{good $2t$-cycle}.

\medskip

\emph{Claim 4. $\ball{17}$ contains neither good $8$-cycles nor good $12$-cycles.}
Suppose not. Let $C$ and $C'$ be the vertex sets of a good 8- or 12-cycle from the definition of good cycles. It follows that either $Q_C$ or $Q_{C'}$ is isomorphic to the graph $C_8^+$ or $C_{12}^+$ in Figure \ref{fig:appendix1} (if some of the vertices on the cycle were of degree 2, $Q_C$ or $Q_{C'}$ could be subgraph of one of these missing some of the degree-1 vertices). By Corollary \ref{cor:appendix}, $C$ or $C'$ increases imbalance of either $(A,B)$ or $(B,A)$.
This proves the claim.

\medskip

Since $\ball{16}$ has no 4-cycles, every 8-cycle in $\ball{16}$ is induced. By using Claim 4, it is easy to see that every 8-cycle can be written as $D=v_1v_2\dots v_8$, where $v_1$ and $v_5$ are 2-adjacent, and $v_2$ and $v_6$ are 2-adjacent. We shall denote such a subgraph by $H_0$ (see Figure \ref{fig:H_0}), and we shall later prove that every copy of $H_0$ in a vicinity of $v_0$ is 2-induced in $G$ (see Claim 8). Prior to that, we need some further properties.

\begin{figure}[htb]
   \centering
   \includegraphics[width=8.5cm]{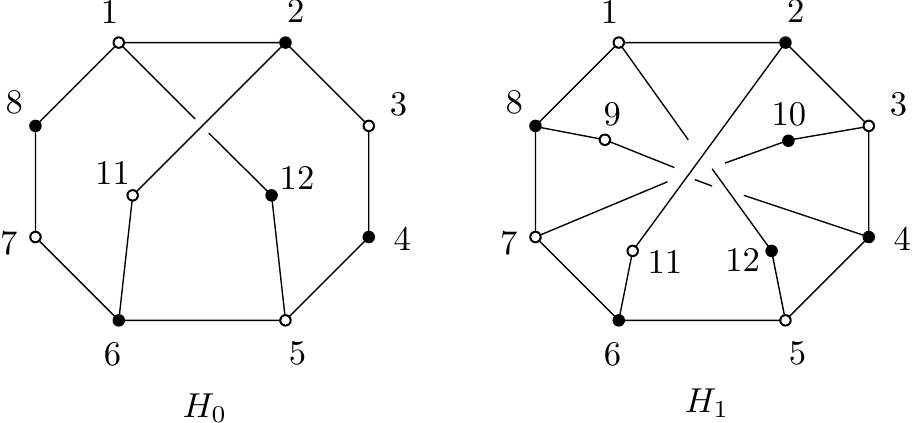}
   \caption{The 2-induced closure of an 8-cycle}
   \label{fig:H_0}
\end{figure}

\medskip

\emph{Claim 5. Every vertex in $\ball{12}$ is contained in a cycle of length 6.}
Suppose $v$ is not contained in a 6-cycle. Since $v\in \ball{12}$, Claims 3 and 3A imply Claim 5 for vertices of degree 2 and their neighbors. Therefore, we may assume that $v$ and its neighbors are of degree 3. Let $C$ be the set consisting of $v$ and the six vertices that are at distance 2 from $v$. Then $Q_C$ is isomorphic to the tree $B_3$ shown in Figure \ref{fig:appendix1} (or its induced subgraph with some of the leaves missing). By Corollary \ref{cor:appendix}, $C$ increases imbalance of either $(A,B)$ or $(B,A)$.

\medskip

We say that an edge of $G$ is \DEF{internal} if it is contained in a 6-cycle in $G$, and is \DEF{external}, otherwise. Claim 5 implies that the set of external edges in the vicinity of $v_0$ form a matching in $G$. By the remarks stated after Claim 4, every 8-cycle in $\ball{17}$ is contained in a copy of the graph $H_0$, which can be written as the union of three 6-cycles. This implies that every edge in the 8-cycle is internal. Thus, external edges cannot belong to cycles of length less than 10. This fact will be used repeatedly in the proof of the next claim which shows that every 6-cycle is incident with at most one external edge.

\begin{figure}[htb]
   \centering
   \includegraphics[width=8.3cm]{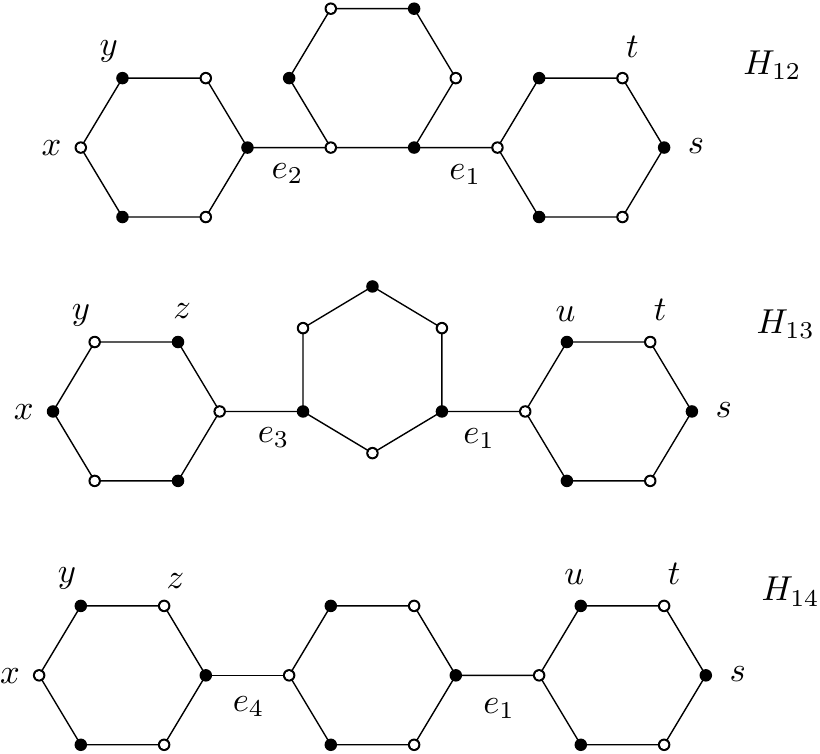}
   \caption{Two external edges at a hexagon}
   \label{fig:external edges}
\end{figure}

\medskip

\emph{Claim 6. Every 6-cycle in $\ball{11}$ is incident with at most one external edge.}
Let $D=v_1v_2\dots v_6$ be a 6-cycle with two or more external edges.
For $i=1,\dots,6$, let $e_i=v_iu_i$ be the edge incident with $v_i$ that is not on the cycle $D$. (We set $e_i=v_i$ if $\deg(v_i)=2$, and we say that $e_i$ is \DEF{internal} in such a case.) Since $u_i\in\ball{12}$, Claim 5 shows that there is a 6-cycle $D_i$ through $u_i$. If $e_i$ is external, the two 6-cycles $D$ and $D_i$ induce a subgraph $H_i$ of $G$ that has only the edge $e_i$ in addition to the two cycles. If there were other adjacencies between $D$ and $D_i$, then $e_i$ would be contained in a cycle of length less than 10 (and would also be contained in $\ball{15}$), which is excluded as argued above. The same argument shows that $H_i$ is 2-induced.

We may assume that $e_1$ is external, and we shall distinguish three cases, depending on whether $e_j$, $j=2,3,4$, is another external edge. Let $H_{1j}$ be the subgraph of $G$ induced on $H_1\cup H_j$. Figure \ref{fig:external edges} shows graphs $H_1\cup H_j$ for $j=2,3,4$. Let us first observe that all vertices shown in Figure \ref{fig:external edges} are distinct in each of the three cases since any identification would give a cycle of length at most 8 through $e_1$. (This is clear for $H_{12}$; similarly in $H_{13}$ and $H_{14}$ where we have to observe that possible identification of $x$ and $s$ in $H_{13}$ or identification of $x$ and $t$ in $H_{14}$ -- or cases symmetric to these -- would force another identification of a neighbor of $x$ and thus yielding a cycle of length at most 8 containing $e_1$.) Let $C$ be the set of vertices of $H_{1j}$ that are depicted as black vertices in Figure \ref{fig:external edges}. Since $D\subseteq \ball{11}$, we have that $C\subseteq \ball{15}$.

Let us first suppose that two of the vertices in $C$ have a common neighbor outside $C$. The only possibilities (up to symmetries) that do not yield a cycle of length at most 8 through $e_1$ are the following:
\begin{enumerate}
\item[(a)]
$s$ and $y$ in $H_{12}$: In this case, there is a good 12-cycle in $\ball{16}$ using the 2-chord from $s$ to $y$, the path from $y$ to $e_2$ that passes through $x$, and the path from $e_2$ to $s$.
\item[(b)]
$s$ and $x$ in $H_{13}$: This case gives a good 12-cycle in $\ball{16}$.
\item[(c)]
$s$ and $z$ in $H_{13}$: This case also gives a good 12-cycle (by using the path $v_3v_4v_5v_6v_1$ through $D$) since any 2-adjacency of white vertices on that 12-cycle would yield a short cycle through $e_1$ or through $e_3$.
\item[(d)]
$s$ and $y$ or $u$ and $y$ in $H_{14}$: Both cases give a good 12-cycle.
\end{enumerate}
Thus we may assume that there are no additional 2-adjacencies between vertices in $C$. If $H_{1j}$ is as shown in Figure \ref{fig:external edges}, i.e., the subgraph shown is induced in $G$, the subgraph $Q_C$ is isomorphic to one of the graphs $\widehat{H}_{12}$, $\widehat{H}_{13}$, and $\widehat{H}_{14}$, shown in Figure \ref{fig:appendix2} in the appendix (with possibly some degree-1 vertices missing if $H_{1j}$ contains vertices whose degree in $G$ is 2). By Corollary \ref{cor:appendix}, $C$ increases imbalance. Thus, we may assume that $H_{1j}$ is not induced. The only possibilities for additional edges (up to symmetries) that do not yield a cycle of length at most 8 through $e_1$ are the following:
\begin{enumerate}
\item[(e)]
The edge $sx$ in $H_{12}$: In this case, $Q_C$ is isomorphic to the graph $\widehat{H}_{12}'$ from the appendix (Figure \ref{fig:appendix2}), and we are done by Corollary \ref{cor:appendix}.
\item[(f)]
The edge $sy$ in $H_{13}$ or the edge $ty$ in $H_{14}$: Each of these cases gives rise to a good 12-cycle (using the path through the vertex $x$).
\item[(g)]
The edge $sx$ in $H_{14}$: This case also gives a good 12-cycle.
\item[(h)]
The edge $sz$ in $H_{14}$:
In this case, $Q_C$ is isomorphic to the graph $\widehat{H}_{12}'$ shown in Figure \ref{fig:appendix2} (where $D$ can be any of the bottom two hexagons), and we are done by Corollary \ref{cor:appendix}.
\end{enumerate}
This completes the proof of Claim 6.

\medskip

\emph{Claim 7. Let $v\in\ball{11}$. If every vertex at distance at most 2 from $v$ has degree 3, then $G$ contains an 8-cycle that has a vertex at distance at most 2 from $v$.}
Let $D=v_1v_2\dots v_6$ be a 6-cycle that contains $v$. By Claim 3A, all vertices in $D$ and all vertices adjacent to $D$ have degree 3. Let $e_i=v_iu_i$ be the edge incident with $v_i$ that is not contained in $D$. By Claim 6, at least five of the edges $e_1,\dots,e_6$ are internal. Each 6-cycle $D'$ containing $e_i$ must use another internal edge $e_j$. If $j=i+2$ or $j=i-2$ (values modulo 6), then $D\cup D'$ contains an 8-cycle, where $v$ is either on the cycle or adjacent to it. Suppose now that $j=i+3$. We may assume that $i=1$ and $j=4$ and that $D'=v_1\dots v_4 v_5'v_6'$. Since any two vertices in $D\cup D'$ lie on a common 6-cycle, at most one of these vertices is incident with an external edge. We may assume that this vertex, if it exits, is $v_6'$. Note that all vertices in $D\cup D'$ have degree 3. Define the edges $e_5',e_6'\notin E(D')$ that are incident with $v_5'$ and $v_6'$, respectively. If a 6-cycle $D''$ through $e_2$ uses any of the edges $e_5,e_6,e_5',e_6'$, then the union $D\cup D'\cup D''$ contains an 8-cycle that is at distance at most 2 from $v$. By symmetry, a similar conclusion holds for all other internal edges leaving $D\cup D'$. Thus, we may assume that the 6-cycle through $e_2$ returns through $e_3$, the cycle through $e_5$ returns through $e_6$, and the 6-cycle leaving $v_5'$ returns through $v_6'$. Denote these 6-cycles by $F_{23},F_{56},F_{56}'$, respectively. Note that any two of these 6-cycles together with a 6-cycle in $D\cup D'$ form a subgraph of $G$ that is isomorphic to the graph $L_3$ shown in Figure \ref{fig:hexagons through internal edges}. It is easy to see that if $L_3$ is not 2-induced, then it contains an 8-cycle that passes through one of the vertices marked $x$ and $y$ in Figure \ref{fig:hexagons through internal edges}. Since $x$ and $y$ are at distance at most 2 from $v$, we get the claim. Thus, we may assume that each of these subgraphs is 2-induced, which implies that also the graph $L_{3,3}$ (see Figure \ref{fig:hexagons through internal edges}) consisting of $D\cup D'$ together with $F_{23},F_{56},F_{56}'$ is 2-induced in $G$. Let $C = V(L_{3,3})\cap A$ and note that $C\subseteq \ball{16}$. Since $L_{3,3}$ is 2-induced in $G$, the subgraph $Q_C$ is isomorphic to the graph $\widehat L_{3,3}$ shown in Figure \ref{fig:appendix3}, and we are done by applying Corollary \ref{cor:appendix}.

\begin{figure}[htb]
   \centering
   \includegraphics[width=11.7cm]{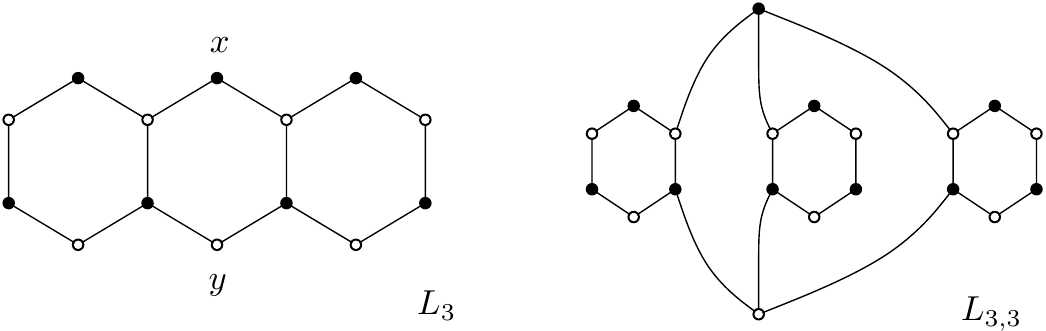}
   \caption{A case of 6-cycles sharing three edges}
   \label{fig:hexagons through internal edges}
\end{figure}

The last case to consider is when every 6-cycle using one of the edges $e_i$ returns to $D$ either through $e_{i-1}$ or through $e_{i+1}$. Suppose first that a 6-cycle $D_1$ through $e_1$ and a 6-cycle $D_3$ through $e_3$ both return through $e_2$. If the union $R=D\cup D_1\cup D_3$ is an $A$-thick (respectively $B$-thick) subgraph of $G$, then either $C=V(R)\cap B$ (resp.\ $C=V(R)\cap A$) gives an imbalance-increasing vertex set in $\ball{16}$ (see Lemma \ref{lem:thick unbalanced}). Therefore, either two vertices in $A$ or two vertices in $B$ are 2-adjacent. The corresponding 2-chord gives rise to two 6-cycles sharing a path of length 2 or 3, which is the case we have already treated above. Even though one of the 6-cycles $D_1$ or $D_3$ may play the role of $D$ in this case, it is still true that a resulting 8-cycle is at distance at most 2 from $v$.

By symmetry, we may now assume that there are precisely three 6-cycles using internal edges $e_1,\dots, e_6$. We may assume that the 6-cycles are $F_{12},F_{34},F_{56}$, where $F_{ij}$ uses the edges $e_i$ and $e_j$. If a vertex $x\in V(F_{12})\setminus V(D)$ and a vertex $y\in V(F_{34})\setminus V(D)$ are either the same, adjacent, or 2-adjacent in $G$, then we obtain two 6-cycles sharing a path of length 2 or 3, which is the case we have already treated above. In that case we get an imbalance-increasing set in $\ball{17}$ or an 8-cycle at distance at most 2 from $v$. Since the same can be said for the other pairs of the 6-cycles $F_{ij}$, we conclude that the graph $R=D\cup F_{12}\cup F_{34}\cup F_{56}$ is 2-induced in $G$. We set $C=V(R)\cap A$ and observe that $Q_C$ is isomorphic to the graph $\widehat{H}_{123}$ shown in Figure \ref{fig:6}. Again, Corollary \ref{cor:appendix} applies. This completes the proof of Claim 7.

\medskip

Let $H_0$ and $H_1$ be the graphs depicted in Figure \ref{fig:H_0}, and let $H_2$, $H_3$, and $H_4$ be the graphs in Figure \ref{fig:H_2 and H_3}. Observe that $H_2$ and $H_3$ are both isomorphic to the Heawood graph with one edge removed and that their subgraphs induced on vertices $1,2,\dots,12$ are isomorphic to $H_1$.

\medskip

\emph{Claim 8. Let $D$ be an 8-cycle in $\ball{10}$. Then $D$ is contained in a 2-induced subgraph isomorphic to the graph $H_0$.}
We have already seen that $D$ must be contained in a subgraph $H$ of $G$ that is isomorphic to $H_0$. This subgraph is induced in $G$ since any chord in $H_0$ gives rise to a cycle of length 4. It remains to see that $H$ is 2-induced.
We will use the notation provided in Figures \ref{fig:H_0} and \ref{fig:H_2 and H_3}, so the vertices of $H$ and other graphs isomorphic to one of $H_0,\dots,H_4$ will be denoted by the integers $1,2,\dots,16$ as shown in the figures.

\begin{figure}[htb]
   \centering
   \includegraphics[width=12cm]{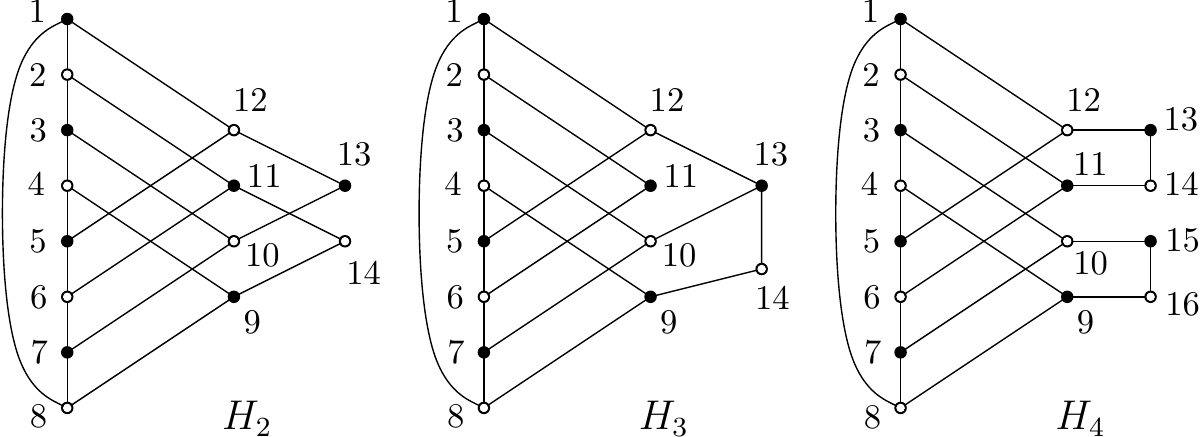}
   \caption{The graphs $H_2$, $H_3$, and $H_4$}
   \label{fig:H_2 and H_3}
\end{figure}

Suppose that $H$ is not 2-induced. Note that any 2-chord starting at vertices 11 or 12 would give rise to a 4-cycle in $G$. Thus, we may assume, by symmetry, that we have a 2-chord 3-10-7. Since this 2-chord is part of another copy of $H_0$, the subgraph obtained from $H$ by adding the 2-chord is induced in $G$. By Lemma \ref{lem:thick unbalanced}, this subgraph cannot be thick, so there is a 2-chord joining two of the vertices 4,8,10,12. The only two pairs not giving a 4-cycle are 4-8 and 10-12. They give rise to isomorphic graphs, so we may assume the 2-chord is 4-9-8. Thus we have a subgraph $H'$ isomorphic to the graph $H_1$ (shown in Figure \ref{fig:H_0}). This subgraph is induced in $G$ since any two vertices in different bipartite classes belong to a common copy of $H_0$ inside $H'$.

Suppose first that $H'$ is not 2-induced. By symmetry between the possible 2-adjacent pairs 9-11 and 10-12, we may assume that 10-13-12 is a 2-chord. By Lemma \ref{lem:thick unbalanced}, two of the vertices 9,11,13 are 2-adjacent. If these are 9 and 11, we obtain a copy of $H_2$. Otherwise, by symmetry, we have a 2-chord 9-14-13 which yields a copy of $H_3$. Let $x$ and $y$ be the degree-2 vertices in the obtained subgraph. If they were adjacent, we would get the Heawood graph, so both graphs are isomorphic to the Heawood graph minus an edge, and thus we may assume henceforth that we have $H_2$. By symmetry between $x$ and $y$ and by symmetry of the bipartite classes $A$ and $B$, we may assume that $x=13\in A$ and that $y=14\in B$. By Claim 3, $x$ and $y$ have degree 3 in $G$ and there are two edges $e=xx'$ and $f=yy'$ going out of $H_2$ in $G$. Since the distance in $H_2$ between $x$ and $y$ is five, both $e$ and $f$ are external edges. Observe that either $x'$ or $y'$ belongs to $\ball{7}$.
This is clear if $v_0\in V(H_2)$. Otherwise, one of these vertices is on a shortest path from $v_0$ to $D$, giving the same outcome. We may assume that $x'\in \ball{7}$.
Let $D'\subset \ball{10}$ be a 6-cycle through the vertex $x'$, and let $C=A\cap (V(H_2)\cup D') \setminus \{1,7,9\}$. Then the subgraph $Q_C$ of $G$ is isomorphic to the graph $\widehat{H}_2^\circ$ shown in Figure \ref{fig:appendix1}. (To check this, note that the 8-cycle in the figure is 2-3-10-13-12-5-6-11-2, where the vertex 11 is adjacent to the vertex $y=14$ having degree 1. The only non-obvious possibility for $Q_C$ being different is that a vertex in $D'$ would be equal to $y'$ and hence adjacent to $y$. However, in that case, $D'$ would be incident with two external edges $e$ and $f$, contradicting Claim 6.)
Now, Corollary \ref{cor:appendix} shows that $C\subset \ball{11}$ increases imbalance.

From now on, we may assume that $H'$ is 2-induced. For vertices $i\in \{9,\dots,12\}$, let $e_i$ be the edge in $G$ incident with vertex $i$ but not contained in $H'$. Note that for any $i,j\in\{9,\dots,12\}$ of different parities, vertices $i$ and $j$ lie on a common 6-cycle in $H'$. Since $D\subset\ball{10}$, we have $H'\subset \ball{11}$. By Claim 6, at most one of the edges $e_9,e_{10}$ is external. By symmetry, we may assume that $e_9$ is internal. Since $H'$ is 2-induced and the distance between 9 and 11 in $H'$ is four, $e_9$ and $e_{11}$ cannot lie on a common 6-cycle. Thus, we may assume (by symmetry between $e_{10}$ and $e_{12}$) that a 6-cycle $D'$ through $e_9$ uses the edge $e_{10}$. In that case, $D'$ is the cycle 9-8-7-10-15-16-9, where 15 and 16 are new vertices. If a 6-cycle through $e_{11}$ uses $e_{10}$, then the cycle is 11-14-15-10-7-6-11 (14 being a new vertex). In this case we obtain a good 8-cycle 11-14-15-16-9-8-1-2-11. Similarly, if a 6-cycle through $e_{12}$ uses $e_9$. Since $e_{11}$ and $e_{12}$ are incident with the same 6-cycle, they cannot be both external by Claim 6. The only possibility for a 6-cycle (excluding previously treated cases) is the 6-cycle 12-13-14-11-6-5-12 (or 12-13-14-11-2-1-12) where 13 and 14 are new vertices. This gives us the graph $H_4$ shown in Figure \ref{fig:H_2 and H_3}. The subgraph $H_4$ is induced in $G$ (or we get a case treated above). If it is not 2-induced, we may assume that we have the 2-chord 13-17-15, and in this case we obtain a good 8-cycle: 12-13-17-15-16-9-8-1-12. Therefore $H_4$ is 2-induced. Now, we take $C=A\cap V(H_4)\setminus\{1\}$. The subgraph $Q_C$ is isomorphic to the graph $\widehat{H}_4^-$ in Figure \ref{fig:appendix1}. Corollary \ref{cor:appendix} shows that $C$ increases imbalance. This exhausts all possibilities and completes the proof of Claim 8.

\medskip

We now define a vertex $\hat v_0$ as follows. If every vertex in $\ball{2}$ is of degree 3 in $G$, then we take $\hat v_0=v_0$. Otherwise, Claim 3A shows that there exists a vertex $\hat v_0\in \ball{3}$ such that all vertices at distance at most 2 from it have degree 3. Claim 7 shows that there exists an 8-cycle $D$ at distance at most 2 from $\hat v_0$. By Claim 8, $D$ is contained in a 2-induced subgraph $H_0$ of $G$, where $H_0$ is depicted in Figure \ref{fig:H_0}. This gives the next conclusion.

\medskip

\emph{Claim 9. There is a 2-induced subgraph of $G$ isomorphic to $H_0$ that contains a vertex in $\ball{4}$.}

\medskip

We shall now fix the subgraph of Claim 9, call it $H_0$ and denote its vertices by integers as in the figure.

\medskip

\emph{Claim 10. All vertices in $H_0$ have degree 3 in $G$, and for $t\le5$, there is no $t$-chord joining a vertex in $\{3,4\}$ with a vertex in $\{7,8\}$.}
By Claim 3, vertices 11 and 12 cannot be of degree 2. If another vertex, which we may assume is the vertex 3, is of degree 2, then let $C=\{1,3,5,7\}$. Since $H_0$ is 2-induced, the corresponding subgraph $Q_C$ is isomorphic to the graph $\widehat{H}_0^=$ from Figure \ref{fig:appendix1}, and we are done by Corollary \ref{cor:appendix}. This proves the first part of the claim.

Suppose next that there is a $t$-chord joining vertex 3 with $\{7,8\}$. By Claim 8, we have $t\ge 3$. Now it is an easy task to verify that the $t$-chord together with a path in $H_0$ gives rise to a good 8-cycle in $G$. This completes the proof of Claim 10.

\medskip

If $i$ is a vertex of degree 2 in $H_0$, then we will denote by $e_i$ the edge in $E(G)\setminus E(H_0)$ that is incident with the vertex $i$. If there is a 6-cycle that contains two of such edges, $e_i$ and $e_j$, then we say that $e_i$ and $e_j$ are \DEF{coupled}. If the corresponding vertices $i$ and $j$ are at distance $t$ in $H_0$, then there is a $(6-t)$-chord $P_{i,j}$ of $H_0$ containing $e_i$ and $e_j$. Since $H_0$ is 2-induced, we know that $t\le3$ in such a case.

\medskip

\emph{Claim 11. If $e_3$ is coupled with $e_{12}$, then the edge $e_4$ is internal.}
Suppose that the path $P_{3,12}$ is 3-13-14-12 (where 13 and 14 are new vertices) and suppose that $e_4=4a$ is an external edge. Since $H_0$ has a vertex in $\ball{4}$, the vertex $a$ lies in $\ball{9}$. By Claim 5, there is a 6-cycle $R=abcdega$ through $a$. Since the edge $4a$ is external, it cannot be contained in a cycle of length less than 10 (see the comment stated before Claim 6). This implies that the cycle $R$ is disjoint from $H_0\cup P_{3,12}$ and its only vertex that is incident to $H_0\cup P_{3,12}$ is $a$. Moreover, none of its vertices is 2-adjacent to a vertex in $H_0\cup P_{3,12}$, except possibly the vertex $d$ that could be 2-adjacent to the vertex 8. Let $C=\{2,4,6,12,13,b,d,g\}$. Then $Q_C$ is isomorphic to the graph $\widehat N_0$ shown in Figure \ref{fig:appendix1}. Now the proof is complete by Corollary \ref{cor:appendix}.

\medskip

\emph{Claim 12. If $e_3$ is coupled with $e_4$ and $P_{3,4}=3abcd4$, then the subgraph $H_5$ induced on $V(H_0\cup P_{3,4})$ is either equal to $H_0\cup P_{3,4}$ or is equal to $H_0\cup P_{3,4}$ together with precisely one of the edges $11c$ or $12b$.}
By Claim 10 and since $G$ has no 4-cycles, the only possible edges of $H_5$ in addition to the edges in $H_0\cup P_{3,4}$ are the two edges $11c$ and $12b$. Thus it remains to see that both of them cannot be present. Suppose, for a contradiction, that they are. By Claim 6 and symmetry, we may assume that the edge $e_a$ leaving $P_{3,4}$ at the vertex $a$ is internal. It is easy to see that its supporting 6-cycle must return to $H_0\cup P_{3,4}$ through the vertex $d$. Let $aefd$ be the corresponding 3-chord. Let $C=\{2,4,6,8,12,a,c,f\}$. By using Claim 10 it is easy to see that the corresponding subgraph $Q_C$ is isomorphic to the graph $\widehat H_5$ in Figure \ref{fig:appendix3}, and we are done by Corollary \ref{cor:appendix}.

\medskip

\begin{figure}[htb]
   \centering
   \includegraphics[width=6.5cm]{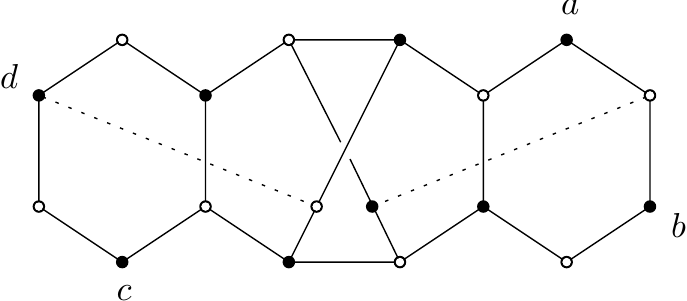}
   \caption{The graphs $H_6^0$, $H_6^1$, and $H_6^2$}
   \label{fig:H_6}
\end{figure}

\emph{Claim 13. If $e_3$ is coupled with $e_4$, then $e_7$ is not coupled with $e_8$.}
Suppose for a contradiction that $e_3$ is coupled with $e_4$ and that $e_7$ is coupled with $e_8$. Let us consider the graph $H_6 = H_0\cup P_{3,4}\cup P_{7,8}$. By Claims 10 and 12, this subgraph is isomorphic to the graph shown in Figure \ref{fig:H_6}, where each of the two dotted edges may or may not be present. If just one of these two edges is present, then we may assume that this is the edge incident with the vertex $d$. By Claim 10, the vertices $a,b,c,d$ cannot be 2-adjacent outside this subgraph, except possibly for $b$ and $d$. Let $C = \{2,4,6,8,12,a,b,c,d\}$. Then $Q_C$ is either isomorphic to one of the graphs $\widehat{H}_6^0$, $\widehat{H}_6^1$, $\widehat{H}_6^2$ in Figure \ref{fig:appendix3}, or to a graph obtained from one of these graphs by identifying the neighbors of vertices $b$ and $d$. The latter possibility does not happen if the left dotted edge in Figure \ref{fig:H_6} is present. By the assumption made above, we may thus assume that $b$ and $d$ are not 2-adjacent when at least one of the dotted edges is present; thus the only additional case arising this way is the graph $\widehat{H}_6^*$ in Figure \ref{fig:appendix3}. In either case, we are done by Corollary \ref{cor:appendix}.

\medskip

\emph{Claim 14. If $e_3$ is coupled with $e_{11}$, then $e_3$ is coupled with $e_4$.}
Let $P_{3,11}$ be the 4-chord $3abc11$. Since $H_0\cup P_{3,11}$ is not thick (Lemma \ref{lem:thick unbalanced}), there is a 2-chord joining two vertices in the larger bipartite class of this graph. By using Claim 10 it is easy to see that the only possibility is a 2-chord joining the vertex $c$ with the vertex 4. This gives the claim.

\medskip

We are now ready to complete the proof. We may assume that $e_3$ is an internal edge, and we may assume that $e_3$ is not coupled with $e_4$ (by using Claim 13 and symmetry). By Claim 14, $e_3$ is not coupled with $e_{11}$ and by Claim 10, it is not coupled with $e_7$ or $e_8$. The only possibility remaining is that it is coupled with $e_{12}$. Claim 11 implies that $e_4$ is an internal edge. The same arguments as used for $e_3$ show that $e_4$ is coupled with $e_{11}$. Now, Claim 10 implies that $e_7$ and $e_8$ cannot be coupled with $e_{11}$ or $e_{12}$. Therefore, $e_7$ is coupled with $e_8$. Take $C$ to be the bipartite vertex class of $H_0\cup P_{3,12}\cup P_{4,11}\cup P_{7,8}$ containing the vertices 1, 3, 5, 7. Then $Q_C$ is isomorphic to the graph $\widehat H_7$ in Figure \ref{fig:appendix3}. Corollary \ref{cor:appendix} applies again, and the proof is complete.
\end{proof}

\appendix

\section{Some subcubic graphs and their eigenvalues}

In the appendix we list a collection of graphs and their critical eigenvalues that were used to obtain balance-increasing vertex sets in the proof of our main theorem.

\begin{figure}[htb]
   \centering
   \includegraphics[width=14cm]{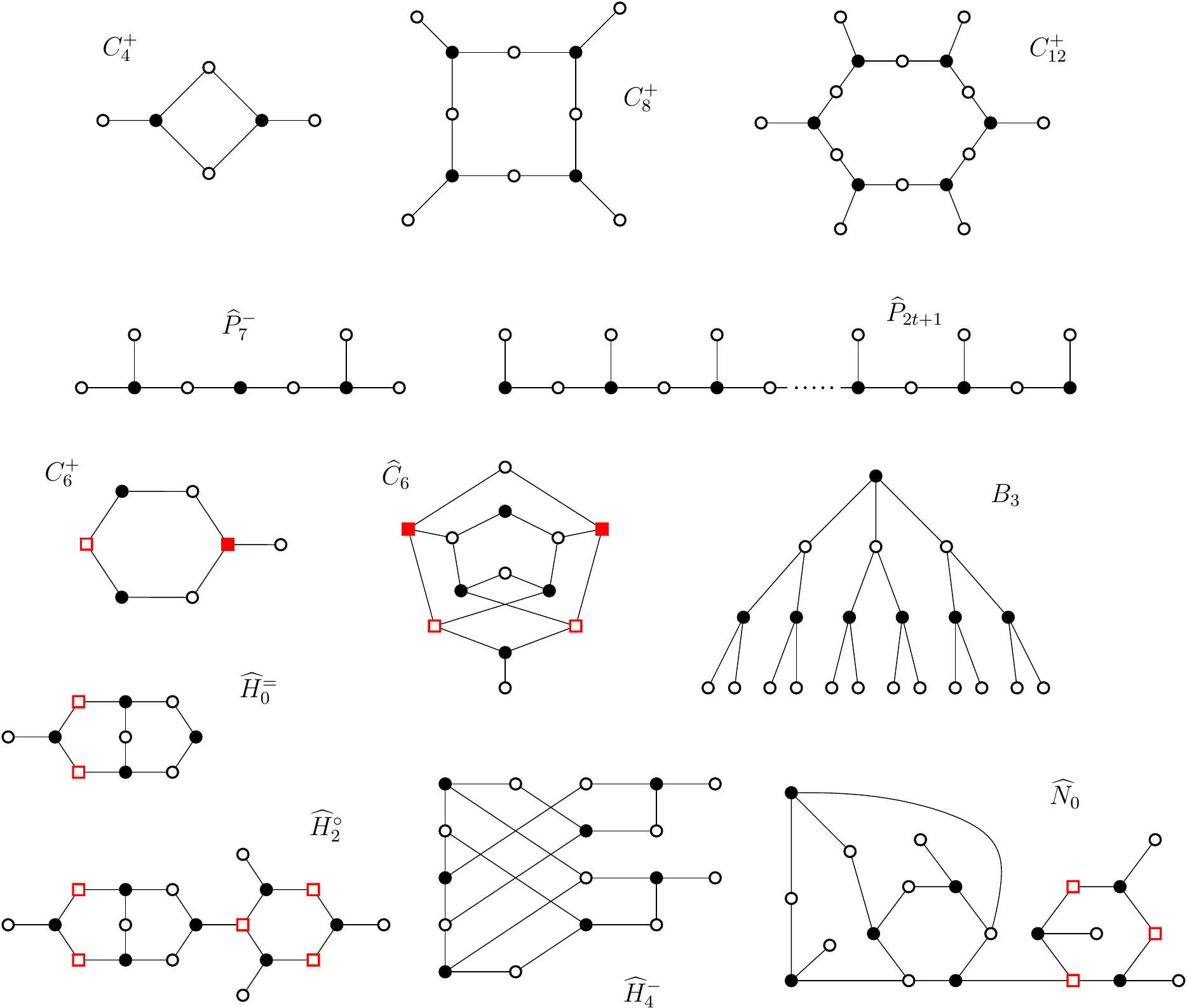}
   \caption{Graphs in Lemma \ref{lem:some graphs}}
   \label{fig:appendix1}
\end{figure}

\begin{figure}[htb]
   \centering
   \includegraphics[width=13cm]{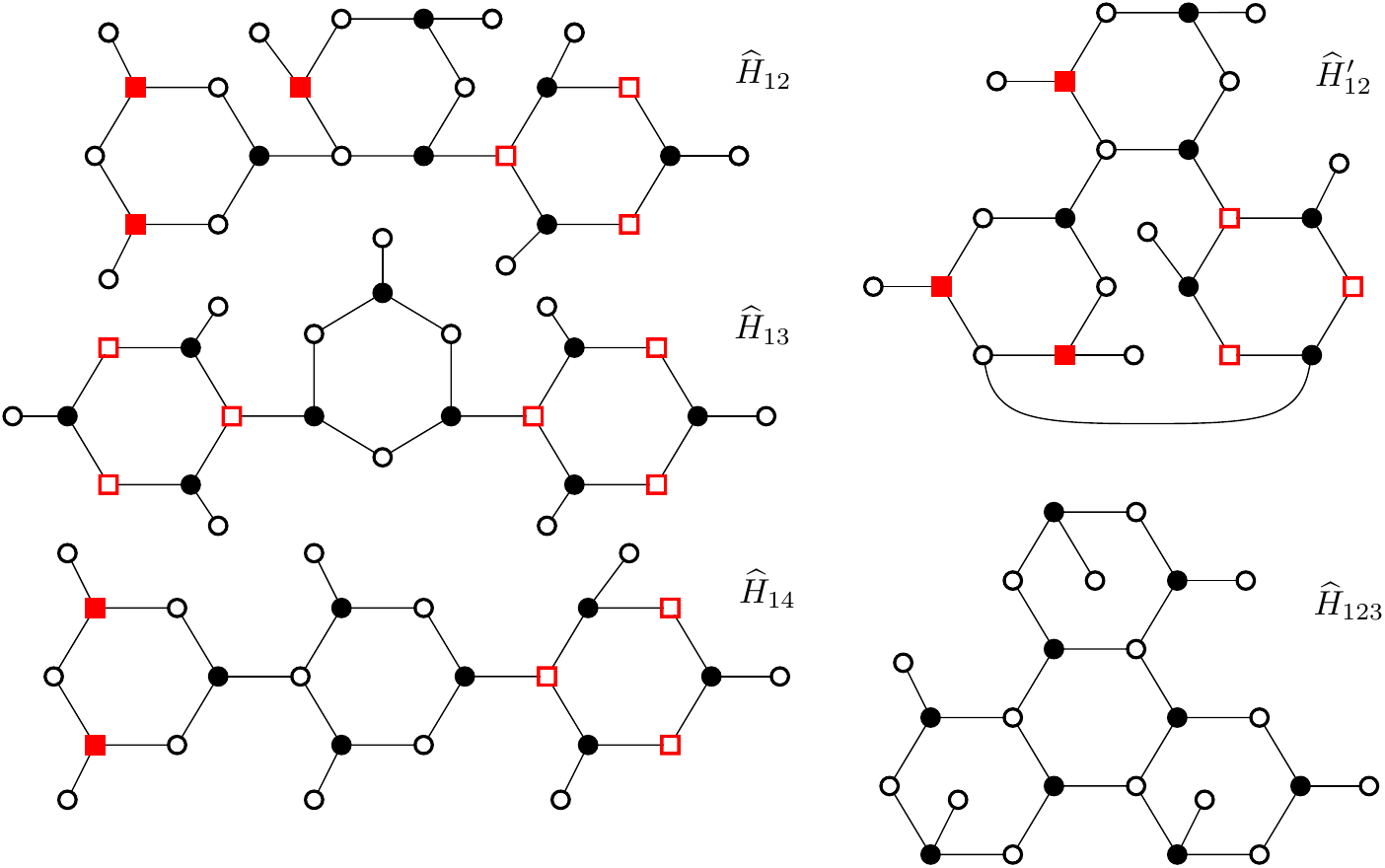}
   \caption{The graphs $\widehat{H}_{12}$, $\widehat{H}_{12}'$, $\widehat{H}_{13}$, $\widehat{H}_{14}$, and $\widehat{H}_{123}$}
   \label{fig:appendix2}
\end{figure}

\begin{lemma}
\label{lem:some graphs}
{\rm (a)}
The graph $C_4^+$ depicted in Fig.~\ref{fig:appendix1} has $\l_2(C_4^+) = 1$.

{\rm (b)}
Let\/ $G$ be one of the graphs $C_6^+$ or $\widehat P_7^-$ depicted in Fig.~\ref{fig:appendix1}. Then $\l_3(G) = 1$.

{\rm (c)}
The graph $C_8^+$ depicted in Fig.~\ref{fig:appendix1} has $\l_4(C_8^+) = 1$.

{\rm (d)}
The graph $C_{12}^+$ depicted in Fig.~\ref{fig:appendix1} has $\l_6(C_{12}^+) = 1$.

{\rm (e)}
The graph $\widehat{C}_6$ depicted in Fig.~\ref{fig:appendix1} has $\l_6(\widehat{C}_6) < 0.91\,$.

{\rm (f)}
The graph $B_3$ depicted in Fig.~\ref{fig:appendix1} has $\l_7(B_3) = 1$.

{\rm (g)} The graphs $\widehat{H}_2^\circ$ and $\widehat{H}_4^-$ depicted in Fig.~\ref{fig:appendix1} have $\l_7(\widehat{H}_2^\circ) = \l_7(\widehat{H}_4^-) = 1$.

{\rm (h)} The graphs $\widehat{H}_0^=$, $\widehat{N}_0$, and $\widehat{H}_5$, depicted in Figs.~\ref{fig:appendix1} and \ref{fig:appendix3} have $\l_4(\widehat{H}_0^=) = 1$, $\l_8(\widehat{N}_0) = 1$, and $\l_8(\widehat{H}_5) < 0.92\,$.

{\rm (i)}
The graphs $\widehat{H}_{12}'$ and $\widehat{H}_{123}$ depicted in Fig.~\ref{fig:appendix2} satisfy $\l_8(\widehat{H}_{12}') = \l_9(\widehat{H}_{12}') = \l_8(\widehat{H}_{123}) = \l_9(\widehat{H}_{123}) = 1$.

{\rm (j)}
The graphs $\widehat{H}_{12}$, $\widehat{H}_{13}$, and $\widehat{H}_{14}$ depicted in Fig.~\ref{fig:appendix2} have
$\l_9(\widehat{H}_{12}) < 0.95$, $\l_9(\widehat{H}_{13}) = 1$, and $\l_9(\widehat{H}_{14}) < 0.96\,$.

{\rm (k)} The graphs $\widehat{H}_6^*$, $\widehat{H}_6^0$, $\widehat{H}_6^1$, $\widehat{H}_6^2$, and $\widehat{H}_7$ depicted in Fig.~\ref{fig:appendix3} have $\l_9(\widehat{H}_6^*) = \l_9(\widehat{H}_6^0) = \l_9(\widehat{H}_6^1) = \l_9(\widehat{H}_6^2) = \l_8(\widehat{H}_6^2) = \l_9(\widehat{H}_7) = 1$.

{\rm (l)}
The graph $\widehat L_{3,3}$ depicted in Fig.~\ref{fig:appendix3} has $\l_{10}(\widehat L_{3,3}) < 0.92\,$.

{\rm (m)}
The graph $\widehat{P}_{2t+1}$ depicted in Fig.~\ref{fig:appendix1} with its horizontal path being of length $2t$ $(t\ge 0)$ has $\l_{t+1}(\widehat{P}_{2t+1}) = 1$.
\end{lemma}

\begin{figure}[htb]
   \centering
   \includegraphics[width=14cm]{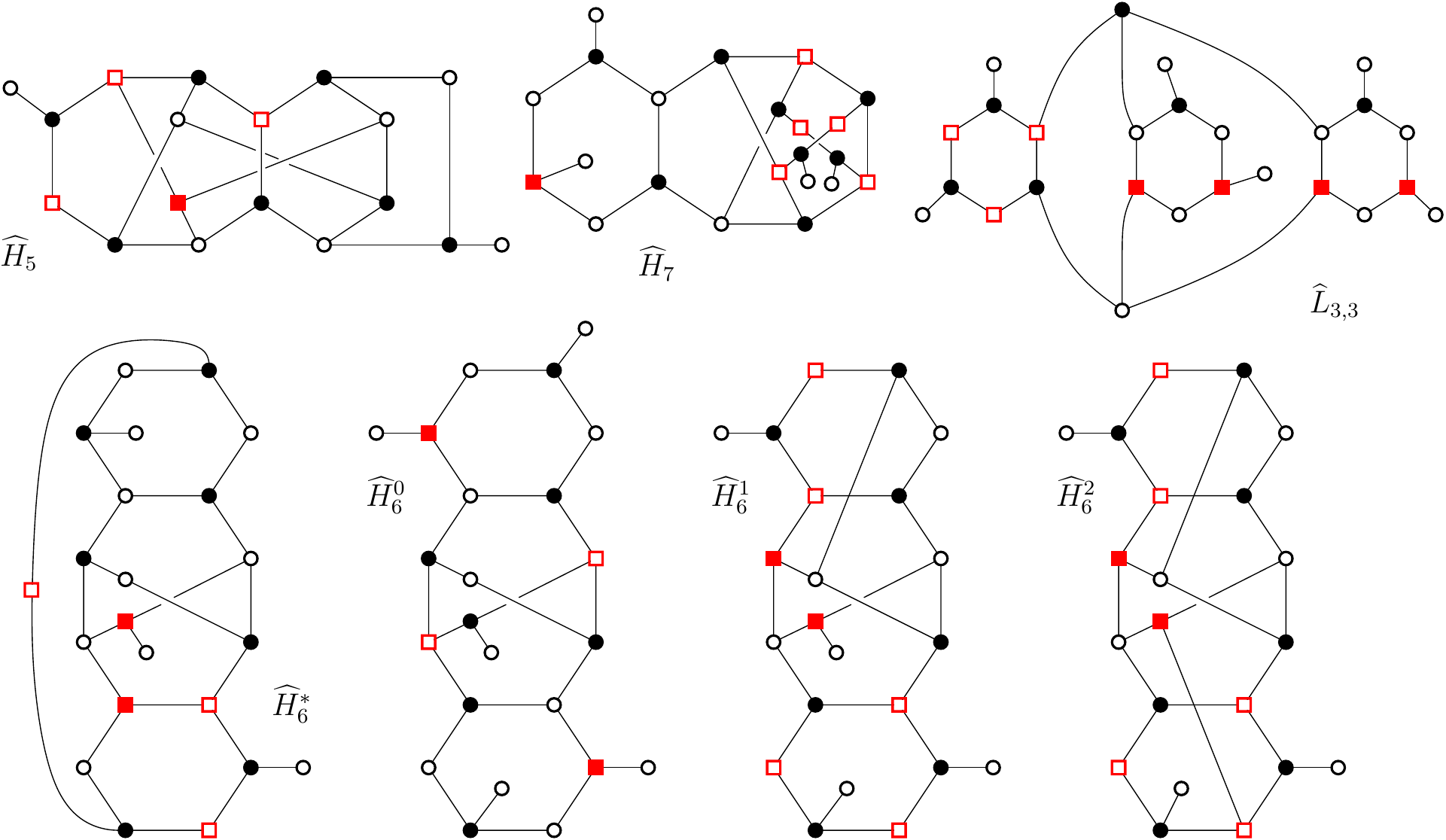}
   \caption{The graphs $\widehat{H}_5$, $\widehat{H}_7$, $\widehat{L}_{3,3}$, $\widehat{H}_6^*$, $\widehat{H}_6^0$, $\widehat{H}_6^1$, and $\widehat{H}_6^2$}
   \label{fig:appendix3}
\end{figure}

\begin{proof}
Claims (a)--(l) were checked by computer. The only graph that needs the proof is $\widehat{P}_{2t+1}$.
Let $v_1v_2\dots v_{2t}v_{2t+1}$ be the horizontal path of length $2t$ in $\widehat{P}_{2t+1}$, and let $C=\{v_2,v_4,\dots,v_{2t}\}$. The subgraph $R=\widehat{P}_{2t+1}-C$ is a matching consisting of $t+1$ disjoint edges. The interlacing theorem shows that $1=\l_{t+1}(R)\le \l_{t+1}(\widehat{P}_{2t+1}) \le \l_1(R) = 1$. This completes the proof.
\end{proof}

Part of the proof of Lemma \ref{lem:some graphs} is based on computer computation. For a reader that may be skeptical about a proof relying on computer evidence, we show in the sequel how to obtain a self-contained proof. We will provide a sketch for a direct proof for all of the cases (in addition to the proof for $\widehat{P}_{2t+1}$ given above) that will suffice to support Corollary \ref{cor:appendix} which is used throughout in Section \ref{sect:improving imbalance}.

Let $G$ be a graph in one of Figs.~\ref{fig:appendix1}--\ref{fig:appendix3}. Let $C$ be the set of vertices of $G$ that are drawn as full circles or full squares. Our goal is to show that $\l_{|C|}(G)\le 1$. Observe that in each case, $G-C$ consists of isolated vertices, thus the Interlacing Theorem implies that $\l_{|C|+1}(G)\le 0$. Thus it suffices to provide evidence that there is an eigenvalue $\l$, where $0<\l\le1$.

For graphs $C_4^+,C_8^+,C_{12}^+,\widehat{P}_7^-$, and $B_3$ we can confirm this by describing an eigenvector for eigenvalue $\l=1$. For $C_4^+,C_8^+,C_{12}^+$, the eigenvector has value $0$ on vertices of degree 2 and value $\pm1$ on other vertices, where each vertex of degree 1 and its neighbor have the same value, and the values $+1$ and $-1$ alternate around the cycle. For $\widehat{P}_7^-$, the vector has values 1 on vertices of degree 1 and 3, value $-1$ on vertices of degree 2 that are adjacent to the degree-3 vertices, and value $-2$ on the vertex in the middle. For $B_3$, the eigenvector has value 3 at the top vertex, values $1$ on adjacent vertices, and value $-1$ at all other vertices.

In the case of $C_6^+$ we can confirm that $\l_3(C_6^+)\le1$ by using the interlacing theorem when we remove the vertex of degree 3 and the vertex that is opposite to it on the 6-cycle. The vertices removed in this case and in the cases treated below are shown as squares. In the case of the graph $\widehat{C}_6$, we remove the two vertices on the left and two vertices on the right of the drawing. The resulting subgraph has one component isomorphic to $C_6$, which has $\l_2(C_6)=1$. Again, the Interlacing Theorem applies. The same happens for graphs $\widehat{H}_0^=$ (where we remove the two square vertices on the left of the drawing) and $\widehat{H}_2^\circ$ (where we remove five vertices). For the graph $\widehat{H}_4^-$ we provide an eigenvector for eigenvalue $1$ in Figure \ref{fig:appendix4}(a). For $\widehat{N}_{0}$ we first remove the three square vertices and then find an eigenvector for eigenvalue 1 for the remaining subgraph (cf.\ Figure \ref{fig:appendix4}(b)).

This exhausts all graphs in Figure \ref{fig:appendix1} and we proceed with those in Figure \ref{fig:appendix2}.

\begin{figure}[htb]
   \centering
   \includegraphics[width=13.5cm]{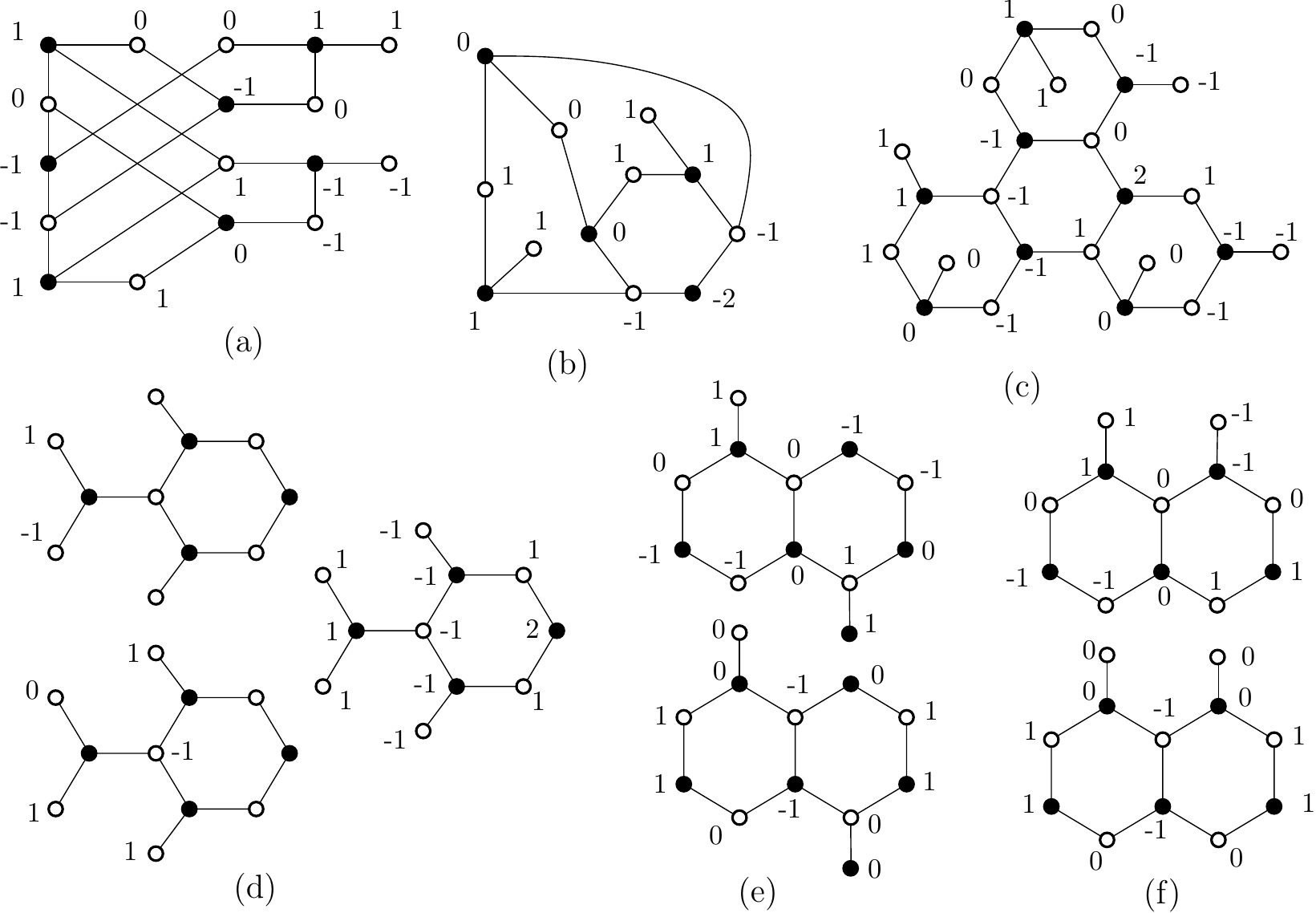}
   \caption{Some graphs and their eigenvectors}
   \label{fig:appendix4}
\end{figure}

For the graphs $\widehat{H}_{12}$ and $\widehat{H}_{12}'$, we remove six square vertices, and
are left with a copy of the graph $\widehat{P}_7^-$ as the non-trivial component. Since $\l_3(\widehat{P}_7^-)=1$, the Interlacing Theorem applies. For the graph $\widehat{H}_{13}$, we can remove its six square vertices, being left with a copy of the graph $C_6^+$ as the only non-trivial component. Again, the Interlacing Theorem applies.

For the graph $\widehat{H}_{14}$, we remove five square vertices. The resulting non-trivial component $Q$ has $\l_4(Q)=1$ and $\l_5(Q)=\l_6(Q)=0$. The evidence for this is shown by three eigenvectors in Figure \ref{fig:appendix4}(d), where the unfilled values for the eigenvectors of 0 are assumed to be 0. (Note that $\l_7(Q)=0$ as well, but this evidence is not needed for our proof.) Finally, the graph $\widehat{H}_{123}$ has eigenvalue 1; its eigenvector is shown in Figure \ref{fig:appendix4}(c).

It remains to treat the graphs in Figure \ref{fig:appendix3}. For the graph $\widehat{H}_5$, we remove the four square vertices. The remaining nontrivial component $Q$ consists of two hexagons sharing an edge plus two additional edges. Figure \ref{fig:appendix4}(e) shows two independent eigenvectors for eigenvalue 1 of $Q$ which implies that $\l_4(Q)\le 1$, and the Interlacing Theorem can be applied.

For the graph $\widehat{H}_7$, we remove its six square vertices. The remaining nontrivial component $Q$ consists of a path $v_1v_2\dots v_6$ with added pendant edges at vertices $v_2,\dots,v_5$. This graph has characteristic polynomial
$$
    \phi(\l) = \l^{10} - 9\l^8 + 24\l^6 - 20\l^4 + 4\l^2.
$$
Note that $\phi(1)=\phi(0)=0$ and that $\phi'(1) > 0$. Basic calculus shows that $\l_5(Q)=0$ and $\l_i(Q)\le1$ for $i=3,4$.

For the graph $\widehat{L}_{3,3}$, the removal of two square vertices from two of the three central hexagons and removal of three square vertices from the third hexagon give a non-trivial component that is isomorphic to $\widehat{P}_7^-$, and we are done as in some of the previous cases. The proof is also easy for $\widehat{H}_6^1$, and $\widehat{H}_6^2$. The removal of indicated seven square vertices leaves only one nontrivial component, which is isomorphic to $C_6$, whose second eigenvalue is 1.

From $\widehat{H}_6^*$ we remove five square vertices, being left with a non-trivial component $Q$ consisting of two adjacent hexagons plus two edges. Figure \ref{fig:appendix4}(f) contains evidence that $Q$ has eigenvalue 1 of multiplicity at least 2, which implies that $\l_4(Q)\le 1$, and interlacing arguments apply.

Finally, in the case of $\widehat{H}_6^0$, removal of four square vertices leaves one non-trivial component, $Q$, which is isomorphic to the graph obtained from the path of length 10 to which we add two pendant edges at each end. The characteristic polynomial $\phi(\l)$ of $Q$ is easily computed:
$$
    \phi(\l) = \l^{15} - 14\l^{13} + 76\l^{11} - 200\l^9 + 259\l^7 - 146\l^5 + 24\l^3.
$$
Then $\phi(1)=0$ and $\phi'(1) = 24 > 0$. Thus, one of $\l_1(Q), \l_3(Q), \l_5(Q), \l_7(Q)$ is equal to 1. However, $\l_7(Q)=0$, so we have that $\l_5(Q)\le1$, and interlacing can be used again. This exhausts all graphs in Figures \ref{fig:appendix1}--\ref{fig:appendix3}.

\subsection*{Acknowledgement}
The author is grateful to Tom Boothby and to Krystal Guo for an independent verification of eigenvalue computations for graphs that appear in the appendix. Computations were done by using both Maple 16
({\tt http://www.maplesoft.com/products/Maple/}) and Sage (System for Algebra and Geometry Experimentation, {\tt http://www.sagemath.org/}).


\end{document}